\documentclass[12pt]{amsart}
\usepackage{amscd,amsmath,amsthm,amssymb,graphics}
\usepackage{pstcol,pst-plot,pst-3d}
\usepackage{lmodern,pst-node}
\usepackage{multicol}
\usepackage{epic,eepic}
\usepackage{amsfonts,amssymb,amscd,amsmath,enumerate,verbatim}
\psset{unit=0.7cm,linewidth=0.8pt,arrowsize=2.5pt 4}

\newpsstyle{fatline}{linewidth=1.5pt}
\newpsstyle{fyp}{fillstyle=solid,fillcolor=verylight}
\definecolor{verylight}{gray}{0.97}
\definecolor{light}{gray}{0.9}
\definecolor{medium}{gray}{0.85}
\definecolor{dark}{gray}{0.6}

\def\NZQ{\Bbb}               
\def\NN{{\NZQ N}}

\def\frk{\frak}               

\def\Phi{{\frk n}}
\def\Phi{{\frk N}}
\def\MI{{\mathcal I}}

\def\MT{{\mathcal T}}

\def\MS{{\mathcal S}}

\def\opn#1#2{\def#1{\operatorname{#2}}} 
\opn\chara{char} \opn\length{\ell} \opn\pd{pd} \opn\rk{rk}
\opn\projdim{proj\,dim} \opn\injdim{inj\,dim} \opn\rank{rank}
\opn\depth{depth} \opn\grade{grade} \opn\height{height}
\opn\embdim{emb\,dim} \opn\codim{codim}

\opn\Tr{Tr} \opn\bigrank{big\,rank}
\opn\superheight{superheight}\opn\lcm{lcm}
\opn\trdeg{tr\,deg}
\opn\reg{reg} \opn\lreg{lreg} \opn\ini{in} \opn\lpd{lpd}
\opn\size{size}\opn\bigsize{bigsize}
\opn\cosize{cosize}\opn\bigcosize{bigcosize}
\opn\sdepth{sdepth}\opn\sreg{sreg}
\opn\link{link}\opn\fdepth{fdepth}
\opn\div{div} \opn\Div{Div} \opn\cl{cl} \opn\Cl{Cl}
\opn\Spec{Spec} \opn\Supp{Supp} \opn\supp{supp} \opn\Sing{Sing}
\opn\Ass{Ass} \opn\Min{Min}\opn\Mon{Mon} \opn\dstab{dstab} \opn\astab{astab}
\opn\Syz{Syz}
\opn\Ann{Ann} \opn\Rad{Rad} \opn\Soc{Soc}
\opn\Im{Im} \opn\Ker{Ker} \opn\Coker{Coker} \opn\Am{Am}
\opn\Hom{Hom} \opn\Tor{Tor} \opn\Ext{Ext} \opn\End{End}
\opn\Aut{Aut} \opn\id{id}

\opn\nat{nat}
\opn\pff{pf}
\opn\Pf{Pf} \opn\GL{GL} \opn\SL{SL} \opn\mod{mod} \opn\ord{ord}
\opn\Gin{Gin} \opn\Hilb{Hilb}\opn\sort{sort}
\opn\aff{aff} \opn\con{conv} \opn\relint{relint} \opn\st{st}
\opn\lk{lk} \opn\cn{cn} \opn\core{core} \opn\vol{vol}
\opn\link{link} \opn\star{star}\opn\lex{lex}
\opn\gr{gr}

\def\pot#1#2{#1[\kern-0.28ex[#2]\kern-0.28ex]}

\opn\dirlim{\underrightarrow{\lim}}
\opn\inivlim{\underleftarrow{\lim}}
\let\iso=\cong

\let\to=\rightarrow

\def\Implies{\ifmmode\Longrightarrow \else
        \unskip${}\Longrightarrow{}$\ignorespaces\fi}
\def\implies{\ifmmode\Rightarrow \else
        \unskip${}\Rightarrow{}$\ignorespaces\fi}
\def\iff{\ifmmode\Longleftrightarrow \else
        \unskip${}\Longleftrightarrow{}$\ignorespaces\fi}

\let\:=\colon
\newtheorem{Theorem}{Theorem}[section]
 
 \newtheorem{Corollary}[Theorem]{Corollary}

\let\epsilon\varepsilon
\let\kappa=\varkappa
\def\qed{\ifhmode\textqed\fi
      \ifmmode\ifinner\quad\qedsymbol\else\dispqed\fi\fi}
\def\textqed{\unskip\nobreak\penalty50
       \hskip2em\hbox{}\nobreak\hfil\qedsymbol
       \parfillskip=0pt \finalhyphendemerits=0}
\def\dispqed{\rlap{\qquad\qedsymbol}}

\opn\dis{dis}
\def\pnt{{\raise0.5mm\hbox{\large\bf.}}}

\opn\Lex{Lex}

\begin{document}
 \title {A note on the regularity of Hibi rings}

 \author {Viviana Ene, J\"urgen Herzog and Sara Saeedi Madani}

\address{Viviana Ene, Faculty of Mathematics and Computer Science, Ovidius University, Bd.\ Mamaia 124,
 900527 Constanta, Romania, and
 \newline
 \indent Simion Stoilow Institute of Mathematics of the Romanian Academy, Research group of the project  ID-PCE-2011-1023,
 P.O.Box 1-764, Bucharest 014700, Romania} \email{vivian@univ-ovidius.ro}

\address{J\"urgen Herzog, Fachbereich Mathematik, Universit\"at Duisburg-Essen, Campus Essen, 45117
Essen, Germany} \email{juergen.herzog@uni-essen.de}

\address{Sara Saeedi Madani, Department of Pure Mathematics, Faculty of Mathematics and Computer Science, Amirkabir University
of Technology (Tehran Polytechnic), 424, Hafez Ave., Tehran 15914, Iran, and
\newline
 \indent School of Mathematics,
Institute for Research in Fundamental Sciences (IPM), P.O. Box 19395-5746, Tehran, Iran} \email{sarasaeedi@aut.ac.ir}

\thanks{The first author was supported by the grant UEFISCDI,  PN-II-ID-PCE- 2011-3-1023.}
\thanks{The paper was written while the third author was visiting the Department of Mathematics of University Duisburg-Essen. She wants to express her thanks for its hospitality.}

 \begin{abstract}
 We compute the regularity of the Hibi ring of any finite distributive lattice in terms of its poset of join irreducible elements.
 \end{abstract}

\thanks{}
\subjclass[2010]{Primary 05E40, 16E05; Secondary 06D99, 06A11.}
\keywords{Distributive lattices, Hibi rings, regularity.}

 \maketitle

\section*{Introduction}
Let $P$ be a finite poset. The  set of poset ideals $L=\MI(P)$, partially ordered by inclusion, is a distributive lattice. According to a classical result of Birkhoff \cite{B}  any finite distributive lattice arises in this way. Now given a field $K$, there is naturally attached to $L$ the $K$-algebra $K[L]$  generated  over $K$ by the elements of $L$ with  defining relations $\alpha\beta-(\alpha\wedge \beta)(\alpha\vee \beta)$ with  $\alpha,\beta\in L$ incomparable. This algebra was introduced by Hibi \cite{Hibi} in 1987 where he showed that $K[L]$ is a Cohen--Macaulay domain with an ASL structure. He also characterized those distributive lattices for which $K[L]$ is Gorenstein. Nowadays $K[L]$ is called the Hibi ring of $L$.

By choosing for each $\alpha\in L$ an indeterminate  $x_\alpha$ one  obtains the presentation $K[L]\iso S/I_L$ where $S$ is the polynomial ring over $K$ in the indeterminates $x_\alpha$ and where $I_L$  is generated by the quadratic binomials $x_\alpha x_\beta- x_{\alpha\wedge \beta}x_{\alpha\vee \beta}$ with $\alpha,\beta\in L$ incomparable. Not so much is known about the graded minimal free $S$-resolution of the toric ideal $I_L$. Of course we know its projective dimension. Indeed, since $K[L]$ is Cohen-Macaulay and since $\dim K[L]$ is known to be equal to $|P|+1$, the Auslander-Buchsbaum formula implies that $\projdim I_L=|L|-|P|-2$. An equally important invariant of a graded module $M$ over a polynomial ring is its Castelnuovo--Mumford  regularity which may be computed in terms of the shifts of the graded minimal free resolution of $M$ and which is denoted by $\reg M$. As a main result of this paper we show that $\reg I_L=|P|-\rank P$. As a consequence we obtain the formula as given in \cite{ERQ} for the regularity of $I_L$ for any planar distributive lattice $L$.  Our result also provides a simple proof for the classification of the distributive lattices for which $I_L$ has a linear resolution, see \cite[Theorem~3.2]{EHH} and \cite[Corollary~10]{ERQ}, and of those lattices for which $I_L$ is extremal Gorenstein, see \cite[Theorem~3.5]{ERQ}.

\section{The regularity of $K[L]$}

Let $P$ be a finite poset. A subset $\alpha\subset P$ is called a {\em poset ideal} of $P$ if whenever $p\in\alpha$ and $q\leq p$, then $q\in \alpha$. We denote by $\MI(P)$ the set of poset ideals of $P$. Note that $\MI(P)$ with the partial order given by inclusion and with union and intersection as join and meet operation is a distributive lattice. Birkhoff's fundamental theorem asserts that any finite distributive lattice $(L,\wedge, \vee)$ arises in this way. To be precise, $L\iso \MI(P)$ where  $P$ is  the subposet of $L$ consisting of all join irreducible elements of $L$. Recall that $\alpha\in L$ is called {\em join  irreducible} if $\alpha\neq \min L$ and  whenever $\alpha=\beta\vee \gamma$, then $\alpha=\beta$ or $\alpha=\gamma$.

Due to this theorem, we may from now on assume $L=\MI(P)$ for some poset $P$. This point of view allows us to interpret  $K[L]$ as a toric ring. Indeed, let $S$ be the polynomial ring over $K$ in the variables $x_\alpha$ with $\alpha\in L$, and let $T$ be the polynomial ring over $K$ in the variables $s$ and $t_p$ with $p\in P$. We consider the $K$-algebra homomorphism $\varphi\: S\to T$ with $\varphi(x_\alpha)=s\prod_{p\in \alpha}t_p$. It is shown in \cite{Hibi} that $I_L=\Ker\varphi$. Thus we see that
\[
K[L]\iso K[\{s\prod_{p\in \alpha}t_p\: \alpha\in L\}]\subset T.
\]
We henceforth identify $K[L]$  with $K[\{s\prod_{p\in \alpha}t_p\: \alpha\in L\}]$.
In \cite[(3.2)]{Hibi} Hibi describes the monomial $K$-basis of $K[L]$: let $\hat{P}$ be the poset obtained from $P$ by adding the  elements $-\infty$ and $\infty$ with $\infty >p$ and $-\infty<p$ for all $p\in P$, and let  $\MS(\hat{P})$ be the set of integer valued functions
\[
v\:\ \hat{P}\to \NN
\]
with $v(\infty)=0$ and $v(p)\leq v(q)$ for all $p\geq q$. Then the monomials
\begin{eqnarray}
\label{basis}
s^{v(-\infty)}\prod_{p\in P}t_p^{v(p)}, \quad v\in \MS(\hat{P})
\end{eqnarray}
form a $K$-basis of $K[L]$. Note that $K[L]$ is standard graded with
\begin{eqnarray}
\label{deg}
\deg(s^{v(-\infty)}\prod_{p\in P}t_p^{v(p)})= v(-\infty).
\end{eqnarray}
Let $\omega_L$ be the canonical ideal of $K[L]$. By using a result of Stanley \cite[pg.\ 82]{St},  Hibi shows in \cite[(3.3)]{Hibi} that the monomials
\begin{eqnarray}
\label{strict}
s^{v(-\infty)}\prod_{p\in P}t_p^{v(p)}, \quad v\in \MT(\hat{P})
\end{eqnarray}
form a $K$-basis of $\omega_L$,  where $ \MT(\hat{P})$ is the set of integer valued functions $v\:\ \hat{P}\to \NN$ with  $v(\infty)=0$ and $v(p)< v(q)$ for all $p> q$.

\medskip
Based on these facts, we are now ready to prove the following

\begin{Theorem}
\label{main}
Let $L$ be a finite distributive lattice and $P$ the poset of join irreducible elements of $L$. Then
\[
\reg I_L=|P|-\rank P.
\]
\end{Theorem}

\begin{proof}
Let $H_{K[L]}(t)$ be the Hilbert series of $K[L]$. Then
\[
H_{K[L]}(t)=\frac{Q(t)}{(1-t)^d},
\]
where $Q(t)=\sum_ih_it^i$ is a polynomial and where $d=|P|+1$ is the Krull dimension of $K[L]$. Since $K[L]$ is Cohen-Macaulay, it follows that $\reg K[L]= \deg Q(t)$.

The $a$-invariant $a(K[L])$ of $K[L]$ is defined to be the degree of the Hilbert series of $K[L]$ (see \cite[Def. 4.4.4]{BHbook}) which by definition is equal to $\deg Q(t)-d$. Thus we see that
\begin{eqnarray}
\label{ainvariant}
\reg  I_L= \reg K[L] +1= a(K[L])+|P|+2.
\end{eqnarray}
On the other hand, following Goto and Watanabe \cite{GW},  who introduced the $a$-invariant,   we have
\[
\label{also}
a(K[L])=-\min\{i\: (\omega_L)_i\neq 0\},
\]
see \cite[Def. 3.6.13]{BHbook}. Thus, since $\rank \hat{P}=\rank P+2$,  the desired formula for the regularity of $K[L]$ follows from (\ref{ainvariant}) once we have shown that $\min\{i\: (\omega_L)_i\neq 0\}=\rank \hat{P}$.

Let $v\in \MT(\hat{P})$ and let $-\infty <p_1< \cdots < p_r<\infty$ be a maximal chain in $\hat{P}$ with $r=\rank P+1.$ Then
\[
0 <v(p_r)<v(p_{r-1})<\cdots < v(p_1)<v(-\infty).
\]
It follows that $v(-\infty)\geq \rank \hat{P}$, and hence (\ref{strict}) implies that  $\min\{i\: (\omega_L)_i\neq 0\}\geq \rank \hat{P}$. In order to prove equality, we consider the depth  function $\delta \: \hat{P}\to \NN$ which  for $p\in \hat{P}$  is defined to be the supremum of the lengths of chains ascending from $p$. Obviously, $\delta\in \MT(\hat{P})$ and $\delta(-\infty)=\rank \hat{P}$. This concludes the proof of the theorem.
\end{proof}

 Recall that $L=\MI(P)$ is called {\em simple} if there is no  $p \in P$ with the property that for every $q \in P$  either $q \leq p$ or
$q \geq p$. In the further discussions we may assume without any restrictions that $L$ is simple, because if we consider the subposet  $P'$ of $P$ which is obtained by removing a vertex $p\in P$ which is comparable with any other vertex of $P$ and let $L'=\MI(P')$,  then $I_L$ and $I_{L'}$ have the same  regularity. Indeed, $|P'|=|P|-1$, and since any maximal chain of $P$ passes through $p$, it also follows that $\rank P'=\rank P-1$. Thus the assertion follows from Theorem~\ref{main}.

\medskip
As an immediate consequence of Theorem~\ref{main}, we get the following characterization of simple distributive lattices whose Hibi rings have linear resolutions, previously obtained in \cite{EHH} and \cite{ERQ}.

\begin{Corollary}
\label{linear resolution}
Let $L$ be a finite simple distributive lattice and $P$ the poset of join irreducible elements of $L$. Then $I_L$ has a linear resolution if and
only if $P$ is the sum  of a chain and an isolated element.
\end{Corollary}

\begin{proof}
The ideal $I_L$ has a linear resolution if and only if $\reg I_L=2$. By Theorem~\ref{main} this is the case if and only if $|P|-\rank P=2$. Say, $\rank P=r$, and let $C=p_0<p_1<\cdots <p_r$ be a maximal chain in $P$. Thus $|P|-\rank P=2$, if and only if there exists a unique  $q\in P$ not belonging to $C$. Suppose $q$ is comparable with some $p_i$. Then $p_i$ is comparable with any other element of $P$, contradiction the assumption that $L$ is simple. Thus if $L$ is simple, then $|P|-\rank P=2$ if and only if $P$ is the sum of the chain $C$ and the isolated element $q$.
\end{proof}

The preceding corollary implies that a finite simple distributive lattice is planar if $I_L$  has a linear resolution. Now let $L$ be any simple planar lattice and $P$ the poset of join irreducible elements of $L$. Then there exist two chains $C_1$ and $C_2$ such that $P$ as a set is the disjoint union of them. We may assume that $|C_1|\geq |C_2|$. It follows from Theorem~\ref{main} that $\reg I_L=|C_2|+1$. This result may also be obtained with the characterization given in \cite[Theorem 4]{ERQ}.

\medskip
We would like to remark that, given a number $k$, Theorem~\ref{main} allows us to determine in a finite number of steps all finite simple distributive lattices $L$ with $\reg I_L=k$. As an example, we consider the case $k=3$. Let $P$ be the poset of join irreducible poset of $L$. By Theorem~\ref{main}, it is enough to find all finite posets $P$ with $|P|-\rank P=3$. Let $C$ be a maximal chain in $P$. Since $|P|=\rank P+3$, it follows that there exist precisely two elements $q,q^\prime\in P$ which do not belong to $C$. The only posets satisfying  $|P|=\rank P+3$ for which $L=\MI(P)$ is  simple are displayed in Figure~\ref{reg=3}.

\begin{figure}[hbt]
\begin{center}
\psset{unit=0.5cm}
\begin{pspicture}(-2,-1)(4,7)
\rput(-12,0){
\rput(0,0){$\bullet$}
\rput(1,0){$\bullet$}
\rput(2,0){$\bullet$}
\rput(2,2){$\bullet$}
\rput(2,4){$\bullet$}
\rput(2,6){$\bullet$}
\psline(2,0)(2,2)
\psline[linestyle=dotted](2,2)(2,4)
\psline(2,4)(2,6)
}
\rput(-7,0){
\rput(0,0){$\bullet$}
\rput(1,0){$\bullet$}
\rput(2,0){$\bullet$}
\rput(2,2){$\bullet$}
\rput(2,4){$\bullet$}
\rput(2,6){$\bullet$}
\psline(2,0)(2,2)
\psline[linestyle=dotted](2,2)(2,4)
\psline(2,4)(2,6)
\psline(1,0)(2,3)
}
\rput(-2,0){
\rput(0,2){$\bullet$}
\rput(0,0){$\bullet$}
\rput(2,0){$\bullet$}
\rput(2,2){$\bullet$}
\rput(2,4){$\bullet$}
\rput(2,6){$\bullet$}
\psline(2,0)(2,2)
\psline[linestyle=dotted](2,2)(2,4)
\psline(2,4)(2,6)
\psline(0,0)(0,2)
}
\rput(3,0){
\rput(0,4){$\bullet$}
\rput(0,0){$\bullet$}
\rput(2,0){$\bullet$}
\rput(2,2){$\bullet$}
\rput(2,4){$\bullet$}
\rput(2,6){$\bullet$}
\psline(2,0)(2,2)
\psline[linestyle=dotted](2,2)(2,4)
\psline(2,4)(2,6)
\psline(0,0)(2,3)
\psline(0,0)(0,4)
}
\rput(8,0){
\rput(0,4){$\bullet$}
\rput(0,0){$\bullet$}
\rput(2,0){$\bullet$}
\rput(2,2){$\bullet$}
\rput(2,4){$\bullet$}
\rput(2,6){$\bullet$}
\psline(2,0)(2,2)
\psline[linestyle=dotted](2,2)(2,4)
\psline(2,4)(2,6)
\psline(0,4)(2,3)
\psline(0,0)(0,4)
}
\rput(13,0){
\rput(0,4){$\bullet$}
\rput(0,0){$\bullet$}
\rput(2,0){$\bullet$}
\rput(2,2){$\bullet$}
\rput(2,4){$\bullet$}
\rput(2,6){$\bullet$}
\psline(2,0)(2,2)
\psline[linestyle=dotted](2,2)(2,4)
\psline(2,4)(2,6)
\psline(0,4)(2,2.5)
\psline(0,0)(2,3.5)
\psline(0,0)(0,4)
}
\end{pspicture}
\end{center}
\caption{}\label{reg=3}
\end{figure}

The Gorenstein ideals $I_L$ with $\reg I_L =3$ are called {\em extremal Gorenstein}. Hibi showed in \cite[pg. 105, d) Corollary]{Hibi} that for any distributive lattice $L$, the ideal $I_L$ is Gorenstein if and only if the poset of join irreducible elements of $L$ is pure. Combining this fact with the above consideration, we recover the result of \cite[Theorem~3.5]{EHH} which says that for a simple distributive lattice $L$, the ideal $I_L$ is extremal Gorenstein if and only if $L$ is one of the lattices shown in Figure~\ref{extremal Gorenstein}.

\begin{figure}[hbt]
\begin{center}
\psset{unit=0.3cm}
\begin{pspicture}(-25.3,-2.5)(4,5)
\rput(5,0){
\psline(-9,3)(-11,1)
\psline(-9,3)(-7,1)
\psline(-11,1)(-9,-1)
\psline(-7,1)(-9,-1)
\psline(-9,3)(-11,5)
\psline(-13,3)(-11,5)
\psline(-13,3)(-11,1)
\psline(-11,1)(-13,-1)
\psline(-13,-1)(-11,-3)
\psline(-11,-3)(-9,-1)

\rput(-9,3){$\bullet$}
\rput(-11,1){$\bullet$}
\rput(-7,1){$\bullet$}
\rput(-9,-1){$\bullet$}
\rput(-11,-3){$\bullet$}
\rput(-13,-1){$\bullet$}
\rput(-13,3){$\bullet$}
\rput(-11,5){$\bullet$}
}

\rput(-5,0){
\psline(-9,3)(-11,1)
\psline(-11,1)(-9,-1)
\psline(-9,3)(-11,5)
\psline(-13,3)(-11,5)
\psline(-13,3)(-11,1)
\psline(-11,1)(-13,-1)
\psline(-13,-1)(-11,-3)
\psline(-11,-3)(-9,-1)

\rput(-9,3){$\bullet$}
\rput(-11,1){$\bullet$}
\rput(-9,-1){$\bullet$}
\rput(-11,-3){$\bullet$}
\rput(-13,-1){$\bullet$}
\rput(-13,3){$\bullet$}
\rput(-11,5){$\bullet$}
}

\rput(18,0){
\psline(-9,3)(-11,1)
\psline(-9,3)(-7,1)
\psline(-11,1)(-9,-1)
\psline(-7,1)(-9,-1)
\psline(-9,3)(-11,5)
\psline(-13,3)(-11,5)
\psline(-13,3)(-11,1)
\psline(-11,1)(-13,-1)
\psline(-13,-1)(-11,-3)
\psline(-11,-3)(-9,-1)
\psline(-15,1)(-13,-1)
\psline(-15,1)(-13,3)

\rput(-9,3){$\bullet$}
\rput(-11,1){$\bullet$}
\rput(-15,1){$\bullet$}
\rput(-7,1){$\bullet$}
\rput(-9,-1){$\bullet$}
\rput(-11,-3){$\bullet$}
\rput(-13,-1){$\bullet$}
\rput(-13,3){$\bullet$}
\rput(-11,5){$\bullet$}
}

\rput(-16,0){
\rput(-11,-3){$\bullet$}
\rput(-13,-0.5){$\bullet$}
\rput(-9,-0.5){$\bullet$}
\rput(-11,-0.5){$\bullet$}
\rput(-13,2.5){$\bullet$}
\rput(-9,2.5){$\bullet$}
\rput(-11,2.5){$\bullet$}
\rput(-11,5){$\bullet$}
\psline(-11,-3)(-13,-0.5)
\psline(-11,-3)(-9,-0.5)
\psline(-11,-3)(-11,-0.5)
\psline(-11,5)(-11,2.5)
\psline(-9,-0.5)(-9,2.5)
\psline(-13,-0.5)(-13,2.5)
\psline(-13,-0.5)(-11,2.5)
\psline(-11,2.5)(-9,-0.5)
\psline(-11,-0.5)(-13,2.5)
\psline(-11,-0.5)(-9,2.5)
\psline(-11,5)(-13,2.5)
\psline(-11,5)(-9,2.5)

}
\end{pspicture}
\end{center}
\caption{}\label{extremal Gorenstein}
\end{figure}

\end{document}